\documentclass[reqno]{amsart}
\usepackage{amssymb,amsmath,amsthm}
\usepackage{color}
\usepackage{enumerate}
\usepackage{fullpage}
\usepackage{url}
\usepackage[Symbol]{upgreek}
\usepackage[unicode,pdfusetitle]{hyperref}
\usepackage[mathscr]{euscript}
\usepackage{mathtools}

\usepackage{tikz}
\usetikzlibrary{decorations.pathmorphing, patterns,shapes}
\usepackage{tikz-cd}

\newtheorem{theorem}{Theorem}
\newtheorem*{theorem*}{Theorem}

\newtheorem*{question*}{Question}

\newtheorem*{conjecture*}{Conjecture}

\newtheorem*{convention*}{Convention}
\newtheorem{claim}{Claim}
\newtheorem*{claim*}{Claim}

\newtheorem*{assumption*}{Assumption}
\newtheorem{corollary}[theorem]{Corollary}
\newtheorem*{corollary*}{Corollary}
\newtheorem{proposition}[theorem]{Proposition}
\newtheorem*{proposition*}{Proposition}

\newtheorem*{lemma*}{Lemma}
\newtheorem{fact}[theorem]{Fact}
\newtheorem*{fact*}{Fact}
\newtheorem{theoremA}{Theorem}

\theoremstyle{definition}
\newtheorem{definition}[theorem]{Definition}
\newtheorem*{definition*}{Definition}
\newtheorem{example}[theorem]{Example}
\newtheorem*{example*}{Example}
\newtheorem{remark}[theorem]{Remark}
\newtheorem*{remark*}{Remark}

\numberwithin{theorem}{section}
\numberwithin{equation}{section}

\DeclareMathOperator{\acl}{acl}

\DeclareMathOperator{\dcl}{dcl}

\DeclareMathOperator{\trdeg}{trdeg}

\DeclareMathOperator{\ring}{ring}

\newcommand{\A}{\mathbb{A}}

\newcommand{\N}{\mathbb{N}}
\newcommand{\PP}{\mathbb{P}}
\newcommand{\Q}{\mathbb{Q}}

\newcommand{\cE}{\mathcal E}

\newcommand{\cL}{\mathcal L}

\newcommand{\cO}{\mathcal O}

\newcommand{\fm}{\mathfrak{m}}

\newcommand{\Ld}{\cL^\delta}

\newcommand{\Td}{T^\delta}
\newcommand{\TdG}{T^\delta_{g}}

\renewcommand{\succeq}{\succcurlyeq}
\renewcommand{\geq}{\geqslant}
\renewcommand{\leq}{\leqslant}
\renewcommand{\phi}{\varphi}

\author{Elliot Kaplan}

\email{ekaplan@mpim-bonn.mpg.de}
\email{elliot.kaplan@umons.ac.be}

\address{Max Planck Institute for Mathematics, Bonn, Germany}
\address{Mathematics Department, University of Mons, Mons, Belgium}

\author{Christoph Kesting}
\email{kestingc@mcmaster.ca}
\address{Department of Mathematics and Statistics, McMaster University, Hamilton, Ontario, Canada}

\title{Generic derivations, differential largeness, and NTP$_2$}
\date{\today}
\begin{document}

\begin{abstract}
We compare Fornasiero and Terzo's framework of generic derivations on algebraically bounded structures with  Le\'on S\'anchez and Tressl's differentially large fields. We show in the case of a single derivation that genericity and differential largeness coincide for \'ez-fields, as introduced by Walsberg and Ye. We also show that an NTP$_2$ algebraically bounded structure remains NTP$_2$ after expanding by a generic derivation.
\end{abstract}
\maketitle

\section{Introduction}
In this note, $\cL$ is a language extending $\cL_{\ring}=\{0,1,+,\cdot\}$, and $K$ is an $\cL$-structure expanding a field of characteristic zero. We let $n,m,k,r$ range over $\N = \{0,1,2,\ldots\}$. 

The last few years have seen new approaches to the model theory of ``tame'' differential fields, often under the guiding principle: ``The model theory of a differential field is largely determined by the model theory of the underlying field, so long as the derivation is generic.'' \emph{Genericity} roughly means that any system of differential equations and inequalities has a solution, so long as the existence of a solution is consistent with the theory of the underlying field.

One recent approach is due to Fornasiero and Terzo~\cite{FT25}, who axiomatized what it means for a derivation on an algebraically bounded expansion of a field to be \emph{generic}. The structure $K$ is \textbf{algebraically bounded} if for all elementary extensions $K^* \succeq_{\cL}K$, all $B \subseteq K^*$, and all $a \in K^*$, we have
\[
a \in \acl_\cL(K\cup B) \iff \trdeg(a|K(B)) = 0.
\]
Many tame classes of fields are known to be algebraically bounded, including real closed fields, algebraically closed fields, and henselian fields of characteristic zero (in the language of rings, though one can also allow a predicate for the valuation ring in the last case). Algebraic boundedness was introduced by van den Dries~\cite{vdD89}, and the equivalence of van den Dries' definition with the one given here is~\cite[Lemma 2.12]{JY23}.

Let $\dim$ denote algebraic dimension on subsets of cartesian powers of $K$, so $\dim X$ is the transcendence degree of the function field $K(X)$ over $K$ for $X \subseteq K^n$. Equivalently, $\dim X = \max_{a\in X^*}\trdeg(a|K)$, where $(K^*,X^*)$ is a $|K|^+$-saturated elementary extension of $(K,X)$; see~\cite[Section 2]{vdD89}. 
If $K$ is algebraically bounded, then $\dim$ is a definable dimension on $K$, meaning that for any definable family $(X_a)_{a \in A}$ and any $r$, the set of $a\in A$ for which $\dim(X_a) = r$ is definable. 

Let $\delta$ be a derivation on $K$. Fornasiero and Terzo define genericity as follows:

\begin{definition}[Genericity]
 The derivation $\delta$ is \textbf{generic} if for all $\cL(K)$-definable $X \subseteq K^{1+r}$, if the projection of $X$ onto the first $r$ coordinates has dimension $r$, then there is $a \in K$ with $(a,\delta a,\ldots,\delta^ra) \in X$.
\end{definition}
When $K$ is algebraically bounded, genericity is first-order axiomatizable. 
Fornasiero and Terzo go on to show that if $K$ is algebraically bounded and $\delta$ is generic, then several model-theoretic properties of $K$ transfer to $(K,\delta)$, including quantifier elimination and model completeness~\cite[Corollaries 3.16 and 3.17]{FT25}.

Another recent research program, initiated by  Le\'on S\'anchez and Tressl, is the study of \emph{differentially large fields}~\cite{LST24}. These fields serve as a differential analog of the \emph{large} fields of Pop; see~\cite{Po14}. Unlike the setting of algebraically bounded structures with generic derivations, differential largeness is defined only for \emph{pure} differential fields (that is, when $\cL = \cL_{\ring}$).

\begin{definition}[Differential largeness]
The underlying differential field $(K,\delta)$ is \textbf{differentially large} if 
\begin{enumerate}
\item  $K$ is large as a field, and 
\item for every differential field extension $(L,\delta) \supseteq (K,\delta)$, if $K$ is existentially closed in $L$ as a field, then $(K,\delta)$ is existentially closed in $(L,\delta)$ as a differential field. 
\end{enumerate}
\end{definition}

What is the relationship between these two notions? We note that there are large, non-algebraically bounded fields~\cite[Example 10]{Fe10}, as well as algebraically bounded, non-large fields~\cite[Example 4.30]{JY25}. Using the characterization of differential largeness provided in~\cite{LST25}, it is not difficult to show that if the underlying field $K$ is large and $\delta$ is generic, then the underlying differential field $(K,\delta)$ is differentially large; see Corollary~\ref{cor:oneway}. As for the converse, an obvious obstruction occurs in the case that there are $\cL$-definable subsets of $K$ that are not definable in the $\cL_{\ring}$-language, as the axiom of  differential largeness can't assert anything about these sets. Here is a concrete example:
\begin{example}
Let $(K,\delta)$ be a differentially closed field and let $C\coloneqq \ker(\delta)$ be the constant field of $K$. Let $t \in K$ be transcendental over $C$ and consider the subfield $C(t)$ of $K$. Let $\cO_t\subseteq C(t)$ be the $t$-adic valuation ring on $C(t)$, and let $\cO \subseteq K$ be a valuation ring lying over $\cO_t$ (we can even find $\cO$ with residue field isomorphic to $C$; see~\cite[Chapter V, Theorem 9]{Ja64}). Then $(K,\cO)$ is an algebraically closed nontrivially valued field. Let $\cL = \cL_{\ring}\cup\{\cO\}$ be the language of valued fields, so  $(K,\cO)$ is algebraically bounded as an $\cL$-structure. Let $\fm$ be the maximal ideal of $\cO$. Then $\dim(\fm) = 1$, but there is no $a\in \fm$ with $\delta a = 0$, since $C \subseteq \cO^\times = \cO \setminus \fm$. Thus, $\delta$ is not generic.
\end{example}
 
In light of this example, a converse only becomes plausible if we assume that $K$ is a large field with no additional structure. 

\begin{question*}
Suppose that $\cL = \cL_{\ring}$. If $(K,\delta)$ is differentially large, then is $\delta$ generic?
\end{question*}

It seems difficult to answer this question without some understanding of the definable sets in $K$. In Proposition~\ref{prop:axioms} below, we give a ``topological'' axiomatization of differential largeness in terms of the \'etale open topology of Johnson, Tran, Walsberg, and Ye~\cite{JTWY24}. Widawski gives a similar axiomatization of differential largeness in his thesis~\cite{Wi24} (see Remark~\ref{rem:widawski}), and we  thank the referee for bringing this to our attention.

In~\cite{WY23}, Walsberg and Ye isolated the class of \textbf{\'ez-fields}---large, algebraically bounded fields that are ``topologically tame'' with respect to the \'etale open topology. For these fields, we have a good enough understanding of definable sets to answer the question positively.

\begin{theoremA}[Corollaries~\ref{cor:oneway} and~\ref{cor:conv}]\label{thma}
Suppose that $\cL = \cL_{\ring}$ and that $K$ is an \'ez-field. Then $\delta$ is generic if and only if $(K,\delta)$ is differentially large.
\end{theoremA}

Theorem~\ref{thma} was known in the case that $K$ is large and model complete (possibly after adding constants), as an easy consequence of~\cite[Proposition 4.8]{LST24}. This in turn uses that differentially large fields can be axiomatized via Tressl's ``uniform companion''~\cite{Tr05}. These fields, along with many others (including henselian fields of characteristic zero and perfect Frobenius fields), are all \'ez~\cite[Theorem C]{WY23}.

In the second part of this note, we focus on model-theoretic transfer theorems for algebraically bounded structures with generic derivations. Suppose that $K$ is algebraically bounded and let $T$ be the $\cL$-theory of $K$. Let $\PP \coloneqq \dcl_{\cL}(\emptyset)$ and fix a derivation $\delta_\PP$ on $\PP$ (one could, for instance, take $\delta_\PP$ to be the zero derivation; when $\PP$ is algebraic over $\Q$, this is the only possibility). Let $\Ld = \cL \cup \{\delta\}$, let $\Td = T+ \text{``$\delta$ is a derivation extending $\delta_\PP$''}$, and let $\TdG = \Td + \text{``$\delta$ is generic''}$. We include a list of previously established transfer theorems below.

\begin{fact}\label{fact:transfer}\
\begin{enumerate}
\item\label{T1} $\TdG$ is consistent and complete.
\item\label{T2} For every $\Ld$-formula $\varphi(x)$ ($x$ is a tuple of variables), there is an $\cL$-formula $\psi(x_0,\ldots,x_r)$ such that
\[
\TdG \models \forall x(\varphi(x) \leftrightarrow \psi(x,\delta x,\ldots,\delta^rx)).
\]
\item\label{T3} If $T$ is model complete, then $\TdG$ is the model companion of $\Td$. 
\item\label{T4} If $T$ eliminates quantifiers, then so does $\TdG$.
\item\label{T5} If $T$ is stable, then so is $\TdG$.
\item\label{T6} If $T$ has NIP, then so does $\TdG$.
\item\label{T7} If $T$ is distal, then so is $\TdG$.
\item\label{T8} If $T$ is simple, then so is $\TdG$.
\item\label{T9} If $\TdG$ eliminates imaginaries, then it is rosy.
\item\label{T10} If $T$ has NSOP$_1$, then so does $\TdG$.
\end{enumerate}
\end{fact}

The first nine parts of Fact~\ref{fact:transfer} were shown by Fornasiero and Terzo; see~\cite{FT25} for~\eqref{T1}--\eqref{T6} and~\cite{FT25b} for~\eqref{T7}--\eqref{T9}. Part~\eqref{T10}, as well as independent proofs of~\eqref{T5},~\eqref{T8}, and~\eqref{T9}, can be established using Le\'on S\'anchez and Mohamed's framework of \emph{derivation-like theories}~\cite{LSM25}. Transfers of neostability properties defined in terms of sequences, like~\eqref{T5}--\eqref{T7}, can be established quite easily using~\eqref{T2}; see~\cite[Proposition 7.1]{ACGZ22} and~\cite[Theorem 6.2]{FT25} for general criteria, as well as the proofs in~\cite{CP23,FK21}. For properties that can be defined in terms of independence relations like~\eqref{T8}--\eqref{T10}, one shows that an independence relation on $T$ can be used to define one on $\TdG$ satisfying the same properties; see~\cite{FT25,LSM25}. A third class of model-theoretic properties consists of \emph{tree properties}, which are defined by consistency-inconsistency patterns. As far as we are aware, the only transfer result for these types of properties without using independence relations is due to Point~\cite{Po18}, who shows that NTP$_2$ transfers for certain classes of topological fields with generic derivations~\cite{GP10}. We generalize this to all algebraically bounded fields. Our method relies on certain facts about NTP$_2$-theories, such as Chernikov's one-variable theorem, but is general enough to be applied to other tree properties, such as the \emph{antichain tree property}~\cite{AKL23}.

\begin{theoremA}[Theorems~\ref{thm:NTP2} and~\ref{thm:NATP}]\label{thmb}
If $T$ has NTP$_2$, then so does $\TdG$. If $T$ has NATP (that is, if $T$ doesn't have the antichain tree property), then so does $\TdG$.
\end{theoremA}

Differentially large fields and generic derivations are both defined in the case of finitely many commuting derivations, and the results in Fact~\ref{fact:transfer} hold in this more general setting. It seems quite plausible that our Theorems~\ref{thma} and~\ref{thmb} also hold in this setting as well, but we do not investigate this here.

\section{Topological axioms and \'ez-fields}
In giving topological axioms for differential largeness, we need the following alternative axiomatization in terms of differential polynomials:

\begin{fact}[{\cite[Theorem 2.8]{LST25}}]\label{fact:altdifflarge}
A differential field $(K,\delta)$ is differentially large if and only if
\begin{enumerate}
\item $K$ is large as a field and
\item for all $r>0$, all $P\in K[X_0,\ldots,X_r]$, and all nonzero $Q\in K[X_0,\ldots,X_{r-1}]$, if there is $x \in K^{1+r}$ with $P(x) = 0$ and $\frac{\partial P}{\partial X_r}(x) \neq 0$, then there is $a \in K$ with 
\[
P(a,\delta a,\ldots,\delta^r a) = 0\neq Q(a,\delta a,\ldots,\delta^{r-1}a).
\]
\end{enumerate}
\end{fact}
Let $V$ be a $K$-variety and let $V(K)$ denote the set of $K$-points of $V$. The \textbf{\'etale open topology} or \textbf{$\cE_K$-topology} on $V(K)$ is the topology with basis given by sets of the form $f(W(K))$ for \'etale morphisms $f\colon W\to V$. Equipping the $K$-points of each $K$-variety with the $\cE_K$-topology, we obtain a \emph{system of topologies}~\cite[Theorem A]{JTWY24}, meaning that morphisms $f\colon V\to W$ between $K$-varieties induce continuous maps $V(K)\to W(K)$ with respect to the $\cE_K$-topologies, and that these induced maps are open (resp.\ closed) embeddings whenever $f$ is an open (resp.\ closed) immersion. The field $K$ is \textbf{large} if and only if the topology on $V(K)$ is non-discrete whenever $V(K)$ is infinite~\cite[Theorem C]{JTWY24}. For our purposes, we can take this as a definition of largeness. An $\cL_{\ring}(K)$-definable set $X \subseteq V(K)$ is \textbf{\'ez} if $X$ is a finite union of \'etale open subsets of Zariski closed subsets of $V(K)$. If $K$ is not large, then any $\cL_{\ring}(K)$-definable set is \'ez. For large fields, the class of \'ez sets is quite well-behaved:

\begin{fact}[{\cite[Theorems A and B(2)]{WY23}}]\label{fact:ez}
Suppose $K$ is large (and perfect, but we always assume characteristic zero in this note). 
\begin{enumerate}
\item\label{E1} The class of \'ez sets is closed under morphisms of $K$-varieties; in particular, all existentially $\cL_{\ring}(K)$-definable sets are \'ez. 
\item\label{E2} For $V$ a smooth irreducible $K$-variety and $X \subseteq V(K)$ a nonempty \'ez set, we have $\dim X = \dim V$ if and only if $X$ has nonempty $\cE_K$-interior in $V(K)$.
\end{enumerate}
\end{fact}

Now let $\delta$ be a derivation on $K$. For a variety $V$, we let $\tau V$ denote the \textbf{prolongation} of $V$, and we let $\pi_V$ denote the projection map $\tau V\to V$; see~\cite{Mo22}. The prolongation is an analog of the tangent bundle $TV$ that takes the derivatives of defining parameters into account:\ if $x \in V(K)$, then $\delta x \in (\tau_xV)(K)$, and when $V$ is defined over the constant field $\ker(\delta)$, the prolongation and tangent bundle coincide. For $a = (a_1,\ldots,a_n) \in K^n$, we let $\delta a \coloneqq (\delta a_1,\ldots,\delta a_n)$, and for $r \in \N$, we let $\nabla^r(a) \coloneqq (a,\delta a,\ldots,\delta^ra) \in K^{(1+r)n}$.

\begin{proposition}\label{prop:axioms}
Suppose that $K$ is a large field and let $\delta$ be a derivation on $K$. The following are equivalent:
\begin{enumerate}
\item\label{L1} For every smooth irreducible $K$-variety $V$ and every \'{e}z set $X \subseteq (\tau V)(K)$, if $\pi_V(X) \subseteq V(K)$ has $\cE_K$-interior, then there is $a \in V(K)$ with $(a,\delta a) \in X$.
\item\label{L2} For every \'ez set $X \subseteq K^{2r}$, if $\pi(X) \subseteq K^r$ has $\cE_K$-interior, then there is $a \in K^r$ with $(a,\delta a) \in X$. 
\item\label{L3} For every \'ez set $X \subseteq K^{r+1}$, if  $\pi(X) \subseteq K^r$  has $\cE_K$-interior, then there is $y \in K$ with $\nabla^r(y) \in X$. 
\item\label{L4} $(K,\delta)$ is differentially large.
\end{enumerate}
\end{proposition}
\begin{proof}
For~\eqref{L1}$\Rightarrow$\eqref{L2}, just take $V = \A^r$. Suppose~\eqref{L2} holds and let $X \subseteq K^{1+r}$ be as in~\eqref{L3}. Consider the morphism $f \colon \A^{1+r}\to \A^{2r}$ given by 
\[
(x_0,\ldots,x_r) \mapsto (x_0,\ldots,x_{r-1},x_1,\ldots,x_r).
\]
Then $f(X) \subseteq K^{2r}$ is \'ez by Fact~\ref{fact:ez}\eqref{E1} and $\pi(f(X)) = \pi(X)$, so~\eqref{L2} gives $a = (a_0,\ldots,a_{r-1}) \in K^r$ with $(a,\delta a) \in f(X)$. For $y \coloneqq a_0$, we have $\nabla^r(y) \in X$.

To see that~\eqref{L3}$\Rightarrow$\eqref{L4}, we use Fact~\ref{fact:altdifflarge}. Let $r>0$, $P\in K[X_0,\ldots,X_r]$, and $Q\in K[X_0,\ldots,X_{r-1}]^{\neq0}$. Suppose there is $x \in K^{1+r}$ with $P(x) = 0$ and $\frac{\partial P}{\partial X_r}(x) \neq 0$. Then $x$ is a smooth $K$-rational point of $V_P$, the zero-locus of $P$, so we may assume that $V_P$ is smooth and irreducible. Take $X\coloneqq V_P(K)\setminus V_Q(K)$, so $X$ is \'ez and $\pi(X) \subseteq K^r$ has $\cE_K$-open interior by Fact~\ref{fact:ez}\eqref{E2}. Then~\eqref{L3} gives $y \in K$ with $\nabla^r(y)  \in X$. 

Finally, suppose that $(K,\delta)$ is differentially large and let $V,X$ be as in~\eqref{L1}. Then $\tau V$ is smooth as well, so we can use~\cite[Theorem B(1)]{WY23} to take smooth irreducible disjoint subvarieties $W_1,\ldots,W_n$ of $\tau V$ and $\cL_{\ring}(K)$-definable $\cE_K$-open subsets $X_i \subseteq W_i(K)$ for each $i$ with $X = X_1\cup\cdots \cup X_n$. Then $\pi_V(X_i)$ is $\cE_K$-open in $V(K)$ for some $i$, so we set $W \coloneqq W_i$ for this $i$ and we replace $X$ with $X_i$. Further shrinking $X$, we may assume that $X$ is a basic $\cE_K$-open subset of $W(K)$, so $X$ is existentially $\cL_{\ring}(K)$-definable. As $X$ is $\cE_K$-open in $W(K)$, we can take an elementary $\cL_{\ring}$-extension $K^* \succeq K$ containing a tuple $(x,u) \in X^*$ that is $K$-generic in $W$ using Fact~\ref{fact:ez}\eqref{E2}. Then $x$ is $K$-generic in $V$ and $(x,u) \in \tau V(K^*)$, so we may extend $\delta$ to a derivation $\delta^*\colon K^*\to K^*$ with $\delta^* x = u$; see~\cite[Chapter IV, Theorems 14 and 18]{Ja64}. As $(K,\delta)$ is differentially large and $K$ is $\cL_{\ring}$-existentially closed in $K^*$, the differential field $(K,\delta)$ is existentially closed in $(K^*,\delta^*)$. As $X$ is existentially $\cL_{\ring}(K)$-definable, we find $a \in V(K)$ with $(a,\delta a) \in X$.
\end{proof}

When $K$ is not large, the conditions in Proposition~\ref{prop:axioms} are trivially equivalent (they never hold). Using~\eqref{L3}$\Rightarrow$\eqref{L4} of Proposition~\ref{prop:axioms} and Fact~\ref{fact:ez}, we have:

\begin{corollary}\label{cor:oneway}
Suppose that $K$ expands a large field and that $\delta$ is generic. Then the underlying differential field $(K,\delta)$ is differentially large.
\end{corollary}

An \textbf{\'ez-field} is by definition a large field for which every definable set is an \'ez set. For these fields, we get the converse:

\begin{corollary}\label{cor:conv}
Suppose that $\cL = \cL_{\ring}$ and that $K$ is an \'ez-field. If $(K,\delta)$ is differentially large, then $\delta$ is generic.
\end{corollary}

\begin{remark}\label{rem:widawski}
Widawski also considers connections between differential largeness and the \'etale open topology in his thesis, introducing the notion of an \emph{\'etale prepared system} and characterizing differential largeness (with several commuting derivations) in terms of the existence of solutions to such a system~\cite[Theorem 5.3.5]{Wi24}. In~\cite[Theorem 5.4.5]{Wi24}, he essentially proves the equivalence \eqref{L3}$\Leftrightarrow$\eqref{L4} in Proposition~\ref{prop:axioms} above, also in the case of a single derivation. We note that our methods give a quick proof of his quantifier elimination result of differentially large Frobenius fields of characteristic zero~\cite[Theorem 6.3.6]{Wi24} in the case of a single derivation: the derivation on such a field is generic by Corollary~\ref{cor:conv} since perfect Frobenius fields are \'ez~\cite[Theorem 7.1]{WY23}, and quantifier elimination follows by~\eqref{T4} of Fact~\ref{fact:transfer}.
\end{remark}

\section{Transferring \texorpdfstring{NTP$_2$}{NTP2} and NATP}

In this section, $K$ is an algebraically bounded structure and $\delta$ is a generic derivation on $K$. We also assume that $(K,\delta)$ is sufficiently saturated.

A formula $\varphi(x,y)$ (where $x,y$ are tuples of variables) has the \textbf{tree property of the second kind} (TP$_2$) if there is an array of tuples $(a_{i,j})_{i,j<\omega}$ such that 
\begin{enumerate}
\item The formula $\phi(x,a_{i,j})\wedge \phi(x,a_{i,j'})$ is inconsistent for all $i$ and all $j< j'$.
\item The partial type $\{\phi(x,a_{i,f(i)}):i <\omega\}$ is consistent for all $f\colon \omega\to \omega$.
\end{enumerate}
A theory $T$ has TP$_2$ if some formula has TP$_2$.

\begin{theorem}\label{thm:NTP2}
If $\TdG$ has TP$_2$, then so does $T$.
\end{theorem}
\begin{proof}
Assume that $\TdG$ has TP$_2$, as witnessed by an $\Ld$-formula $\phi(x,y)$ and an array of tuples $(a_{i,j})_{i,j<\omega}$. By \cite[Lemma 3.2]{Ch14}, we may also assume that  $|x| = 1$. By Fact~\ref{fact:transfer}\eqref{T2}, the formula $\phi$ is equivalent to a formula of the form $\psi(\nabla^{r}x,\nabla^{s}y)$ for natural numbers $r,s$ and an $\cL$-formula $\psi$. By replacing $a_{i,j}$ by $\nabla^{s}(a_{i,j})$ and augmenting $y$, we may assume that $s = 0$. For the rest of this proof, we fix $r$ and an $\cL$-formula $\psi(x_0,\ldots,x_r,y)$ (each $x_i$ unary) such that $\psi(\nabla^{r}x,y)$ has TP$_2$. We assume that $r$ is minimal with this property. Note that if $r = 0$, then $\psi$ is an $\cL$-formula, so $T$ has TP$_2$ and we are done. Thus, we assume for the remainder of the proof that $r>0$. We fix an array  $(a_{i,j})_{i,j<\omega}$ witnessing TP$_2$, and we may arrange that this array is \emph{strongly $\Ld$-indiscernible}, meaning that each row $(a_{i,j})_{j<\omega}$ is $\Ld$-indiscernible over the other rows and that the sequence of rows is $\Ld$-indiscernible; see~\cite[Lemma 5.6]{KKS14}.

For $i,j<\omega$, let 
\[
X_{i,j} \coloneqq \{(x_0,\ldots,x_r) \in K^{r+1}:K \models \psi(x_0,\ldots,x_r,a_{i,j})\},\qquad X_{i,j}^\nabla\coloneqq \{x \in K:\nabla^{r}x \in X_{i,j}\}.
\]
Then $X_{i,j}^\nabla \cap X_{i,j'}^\nabla= \emptyset$ for all $i$ and $j\neq j'$, but $\bigcap_i X_{i,f(i)}^\nabla$ is nonempty for all $f \colon \omega \to \omega$. Let $\pi\colon K^{r+1}\to K^r$ be the projection map onto the first $r$ coordinates.
\begin{claim}\label{claim1}
The projection $\pi(X_{i,j})$ has dimension $r$ for all $i,j$.
\end{claim}
\begin{proof}[Proof of Claim~\ref{claim1}]
Suppose not, so we find a polynomial $P(x_0,\ldots,x_{r-1},y)$ such that $P(x_0,\ldots,x_{r-1},a_{i,j})$ is not identically zero but vanishes on $\pi(X_{i,j})$. In particular, $P(\nabla^{r-1}x,a_{i,j}) = 0$ for all $x \in X_{i,j}^\nabla$. This yields a rational function $Q$ such that $\delta^r x = Q(\nabla^{r-1}x,\nabla^{r}a_{i,j})$ for all $x \in X_{i,j}^\nabla$. Thus, $\psi(\nabla^{r}x,a_{i,j})$ is equivalent to the formula 
\[
\psi(\nabla^{r-1}x,Q(\nabla^{r-1}x,\nabla^{r}a_{i,j}),a_{i,j})
\]
contradicting minimality of $r$.
\end{proof}
\begin{claim}\label{claim2}
Let $f\colon \omega\to \omega$ and let $n>0$. Then the set
\[
\pi(X_{0,f(0)}\cap X_{1,f(1)}\cap \cdots \cap X_{n-1,f(n-1)})
\]
has dimension $r$. 
\end{claim}
\begin{proof}[Proof of Claim~\ref{claim2}]
For each $i,j<\omega$, set
\[\textstyle
b_{i,j} \coloneqq (a_{ni+k,j+f(k)})_{k<n},\qquad 
\theta(x,b_{i,j}) \coloneqq  \bigwedge_{k<n}\psi(\nabla^r x,a_{ni+k,j+f(k)}).
\]
Then the formula $\theta$ has TP$_2$, as witnessed by $(b_{i,j})_{i,j<\omega}$.  The claim  follows by minimality of $r$ and the previous claim.
\end{proof}
For each $i$ and each $n>0$, we set
\[
F_{i,n}\coloneqq \pi(X_{i,0}\cap X_{i,1}\cap \cdots \cap X_{i,n}),
\]
and we let $Z_{i,n}$ denote the Zariski closure of $F_{i,n}$. Then $Z_{0,0} \supseteq Z_{0,1}\supseteq Z_{0,2}\supseteq\cdots$, so Notherianity of the Zariski topology gives $m$ with $Z_{0,m} = Z_{0,n}$ for $n \geq m$. Note that $Z_{0,m}$ is then $\cL(a_{0,0},\ldots,a_{0,m})$-definable, and we let $Z_{i,m}$ denote the corresponding $\cL(a_{i,0},\ldots,a_{i,m})$-definable Zariski closed set, so $Z_{i,m} = Z_{i,n}$ for $n\geq m$ by indiscernibility. We note that 
\begin{equation}\label{eq:intersection1}
\pi(X_{i,j_0}\cap X_{i,j_1} \cap \cdots \cap X_{i,j_m}) \subseteq Z_{i,m}
\end{equation}
for all $i$ and all $m<j_0<j_1<\cdots <j_m$. Indeed, let $Z$ be the Zariski closure of $\pi(X_{i,j_0}\cap X_{i,j_1} \cap \cdots \cap X_{i,j_n})$. If $Z \not\subseteq Z_{i,m}$, then $Z\cap Z_{i,m}$ is a proper Zariski closed subset of $Z$, and thus of $Z_{i,m}$ as well by indiscernibility of the sequence $(a_{i,j})_{j<\omega}$, but this intersection contains $Z_{i,j_m}$, contradicting our choice of $m$.

Note that $\pi(X_{i,1}\cap X_{i,2})$ has dimension $<r$ for each $i$; if not, then the axioms of $\TdG$ would give $a \in K$ with $\nabla^ra \in  X_{i,1}\cap X_{i,2}$, contradicting that $X_{i,1}^\nabla\cap X_{i,2}^{\nabla}= \emptyset$. Thus, $m>0$ and $Z_{i,m}$ has dimension $<r$ for each $i$. Now for each $i,j$, let $c_{i,j}\coloneqq (a_{i,0},\ldots,a_{i,m},a_{i,j+m+1})$, 
and set 
\[
Y_{i,j} \coloneqq\{(x_0,\ldots,x_r)\in X_{i,j+m+1}:(x_0,\ldots,x_{r-1})\not\in Z_{i,m}\}
\]
Then $Y_{i,j}$ is defined by some $\cL$-formula $\theta(x_0,\ldots,x_r,c_{i,j})$, and we claim that this formula has $(m+1)$-TP$_2$, meaning that 
\[
Y_{i,j_0}\cap \cdots \cap Y_{i,j_m} = \emptyset
\]
for all $i$ and all $j_0<\cdots<j_m$, but that the intersection $\bigcap_{i}Y_{i,f(i)}$ is nonempty for all $f\colon \omega\to \omega$. The first part follows by~\eqref{eq:intersection1}. For the second part, let  $f\colon \omega \to \omega$ and let $n$ be given. Then the projection 
\[\textstyle
\pi\big(\bigcap_{i<n}Y_{i,f(i)}\big) = \pi\big(\bigcap_{i<n}X_{i,f(i)}\big) \setminus \bigcup_{i<n}Z_{i,m}
\]
has dimension $r$ by Claim~\ref{claim2}. In particular, this intersection is nonempty, so $\bigcap_{i<\omega}Y_{i,f(i)}$ is nonempty as well. By~\cite[Proposition 5.7]{KKS14}, $T$ has TP$_2$, as witnessed by some finite conjunction of the formula $\theta$.
\end{proof}

This approach works for other tree properties that can be reduced to one variable and witnessed by strongly indiscernible parameters. For completeness, we show how to adapt our approach to the antichain tree property, introduced in~\cite[Definition 4.1]{AK24}.

A formula $\phi(x,y)$ has the antichain tree property (ATP) if there exists a tree-indexed set of parameters $(a_\eta)_{\eta \in 2^{<\omega}}$ such that 
\begin{enumerate}
\item The formula $\phi(x,a_\eta)\wedge \phi(x,a_\nu)$ is inconsistent whenever $\eta$ is a strict truncation of $\nu$.
\item The partial type $\{\phi(x,a_\eta) : \eta \in A\}$ is consistent for any antichain $A \subseteq 2^{<\omega}$.
\end{enumerate}
A theory has ATP if there is a formula that has ATP. We say that $T$ has \textbf{NATP} if it does not have ATP. Any ATP theory is  TP$_2$ and SOP$_1$; see~\cite[Propositions 4.4 and 4.6]{AK24}. 

For $\eta,\nu \in 2^{<\omega}$, we write $\eta \lhd\nu$ to indicate that $\eta$ is a strict truncation of $\nu$, and we write $\eta^\frown \nu$ to denote the concatenation of $\eta$ and $\nu$. Given also $A \subseteq 2^{<\omega}$, we put $\eta^\frown A \coloneqq \{\eta^\frown \nu:\nu\in A\}$.

\begin{theorem}\label{thm:NATP}
If $\TdG$ has ATP, then so does $T$.
\end{theorem}
\begin{proof}
Assume that $\TdG$ has ATP. Using~\cite[Fact 2.5 and Theorem 3.17]{AKL23} and arguing as in the proof of Theorem~\ref{thm:NTP2}, we may assume that this is witnessed by a formula $\psi(\nabla^r x,y)$ where $x$ is unary, $\psi(x_0,\ldots,x_r,y)$ is an $\cL$-formula, and  $r$ is minimal, along with a strongly indiscernible tree-indexed set of parameters $(a_\eta)_{\eta \in 2^{<\omega}}$; see~\cite[Definition 2.4]{AKL23} for the precise definition of strong indiscernibility. As before, let 
\[
X_{\eta} \coloneqq \{(x_0,\ldots,x_r) \in K^{r+1}:K \models \psi(x_0,\ldots,x_r,a_{\eta})\},\qquad X_{\eta}^\nabla\coloneqq \{x \in K:\nabla^{r}x \in X_{\eta}\}
\]
for $\eta \in 2^{<\omega}$, and let $\pi\colon K^{r+1}\to K^r$ be the projection map onto the first $r$ coordinates. 
By the proof of Claim~\ref{claim1} above, $\pi(X_{\eta})$ has dimension $r$ for all $\eta$. Obtaining an analog of Claim~\ref{claim2} takes a bit more work:
\begin{claim*}
Let  $A\subseteq 2^{<\omega}$ be a finite nonempty antichain. Then the set $\pi\big(\bigcap_{\eta \in A}X_{\eta}\big)$ has dimension $r$. 
\end{claim*}
\begin{proof}
Fix $\nu \in A$. For $\eta =\langle i_0,i_1,\ldots,i_{m-1}\rangle\in 2^{<\omega}$, we set
\[
\nu^\eta \coloneqq \nu^\frown\langle i_0\rangle^\frown \nu^\frown\langle i_1\rangle^\frown \cdots^\frown \nu^\frown\langle i_{m-1}\rangle,\qquad A_\eta \coloneqq (\nu^\eta)^\frown A.
\]
Then $A_{\emptyset} = A$ and for $B \subseteq 2^{<\omega}$, the set $\bigcup_{\eta\in B}A_\eta$ is an antichain if and only if $B$ is an antichain. Set 
\[\textstyle
b_\eta \coloneqq (a_\mu)_{\mu \in A_\eta},\qquad 
\theta(x,b_\eta) = \bigwedge_{\mu \in A_\eta}\psi(\nabla^r x,a_\mu).
\]
Then the formula $\theta$ has ATP, as witnessed by $(b_\eta)_{\eta \in 2^{<\omega}}$. We conclude by minimality of $r$ as above.
\end{proof}

For $n \in \N$, we set $n\langle 0\rangle \coloneqq \langle 0,0,\ldots,0\rangle\in 2^{n}$ (so $0\langle 0\rangle =\emptyset$). We set $F_{n}\coloneqq \pi\big(\bigcap_{i\leq n}X_{i\langle0\rangle}\big)$, we let $Z_{n}$ denote the Zariski closure of $F_{n}$, and we take $m$ with $Z_m = Z_n$ for $n \geq m$. Arguing as in the proof of Theorem~\ref{thm:NTP2}, we have that $m>0$, that $\dim(Z_m)<r$, and that
\begin{equation}\label{eq:intersection2}
\pi(X_{\eta_0}\cap X_{\eta_1} \cap \cdots \cap X_{\eta_m}) \subseteq Z_m
\end{equation}
for all $m\langle 0\rangle\lhd \eta_0\lhd \eta_1\lhd\cdots \lhd\eta_m$ (this uses that $(b_\eta)_{\eta \in C}$ and $(b_\nu)_{\nu \in C'}$ have the same $\Ld$-type for any two finite chains $C,C'\subseteq 2^{<\omega}$ of the same length, as a consequence of strong indiscernibility). Now for each $\nu\in 2^{<\omega}$, let $c_{\nu}\coloneqq (a_{\emptyset},\ldots,a_{m\langle0\rangle},a_{m\langle0\rangle^\frown\nu})$ and set 
\[
Y_{\nu} \coloneqq\{(x_0,\ldots,x_r)\in X_{m\langle0\rangle^\frown\nu}:(x_0,\ldots,x_{r-1})\not\in Z_{m}\}.
\]
Then $Y_{\nu}$ is defined by an $\cL$-formula $\theta(x_0,\ldots,x_r,c_\nu)$. Arguing as in the proof of Theorem~\ref{thm:NTP2}, using~\eqref{eq:intersection2} and the Claim, we see that $\theta$ has $m$-ATP, meaning that $\bigcap_{\nu \in A}Y_\nu\neq \emptyset$ for any antichain $A$, but that $\bigcap_{\nu \in C}Y_\nu = \emptyset$ for any chain $C$ of length $m$. By~\cite[Lemma 3.20]{AKL23}, $T$ has ATP. 
\end{proof}

\subsection*{Acknowledgments}
The first author is a postdoctoral researcher of the Fonds de la Recherche Scientifique -- FNRS, and was supported in part by the National Science Foundation under Award No.\ DMS-2103240. Parts of this research were conducted while the first author was hosted by the Max Planck Institute for Mathematics, and he thanks the MPIM for its support and hospitality. We thank the referee for helpful feedback.

\providecommand{\bysame}{\leavevmode\hbox to3em{\hrulefill}\thinspace}
\providecommand{\MR}{\relax\ifhmode\unskip\space\fi MR }
\providecommand{\MRhref}[2]{%
  \href{http://www.ams.org/mathscinet-getitem?mr=#1}{#2}
}
\providecommand{\href}[2]{#2}

%\bibliographystyle{amsplain}	
%\bibliography{diff}
\end{document}